\documentclass[12pt,a4paper]{article}

\usepackage{amsmath,amsfonts,amssymb,amsthm}
\allowdisplaybreaks
\usepackage{enumitem}
\usepackage{bbold}
\usepackage{graphicx}
\usepackage{a4wide}
\usepackage{cite}
\usepackage{float}
\usepackage{authblk}

\expandafter\ifx\csname pdfglyphtounicode\endcsname\relax\else
\expandafter\ifx\csname pdfgentounicode\endcsname\relax\else
\input glyphtounicode.tex
\fi
\fi

\renewcommand{\leq}{\leqslant}

\renewcommand{\geq}{\geqslant}

\makeatletter
\newcommand*{\bigcorr@macro}[2]{\sbox{0}{\mbox{$#1($}}\dimen0=\ht0
                \advance\dimen0 by \dp0
                \multiply\dimen0 by #2 \divide\dimen0 by 100}
\newcommand*{\bigcorr@big}[2]{\mbox{$#1\left#2\bigcorr@macro{#1}{85}\vrule
                   height \dimen0 depth 0pt width 0pt\right.\n@space$}}
\newcommand*{\bigcorr@Big}[2]{\mbox{$#1\left#2\bigcorr@macro{#1}{115}\vrule
                   height \dimen0 depth 0pt width 0pt\right.\n@space$}}
\newcommand*{\bigcorr@bigg}[2]{\mbox{$#1\left#2\bigcorr@macro{#1}{145}\vrule
                   height \dimen0 depth 0pt width 0pt\right.\n@space$}}
\newcommand*{\bigcorr@Bigg}[2]{\mbox{$#1\left#2\bigcorr@macro{#1}{175}\vrule
                   height \dimen0 depth 0pt width 0pt\right.\n@space$}}
\DeclareRobustCommand*{\big}[1]{{\mathpalette\bigcorr@big{#1}}}
\DeclareRobustCommand*{\Big}[1]{{\mathpalette\bigcorr@Big{#1}}}
\DeclareRobustCommand*{\bigg}[1]{{\mathpalette\bigcorr@bigg{#1}}}
\DeclareRobustCommand*{\Bigg}[1]{{\mathpalette\bigcorr@Bigg{#1}}}
\DeclareRobustCommand*{\bi}[1]{{#1}}
\DeclareRobustCommand*{\bil}[1]{\mathopen{\bi{#1}}}

\DeclareRobustCommand*{\bir}[1]{\mathclose{\bi{#1}}}
\makeatother

\newtheorem{lemma}{Lemma}
\newtheorem{proposition}{Proposition}
\newtheorem{theorem}{Theorem}
\newcommand{\1}{{\mathbb{1}}}
\newcommand{\NN}{{\mathbb{N}}}
\newcommand{\RR}{{\mathbb{R}}}
\newcommand{\UU}{{\mathbb{U}}}
\newcommand{\dd}{{\mathrm{d}}}
\newcommand{\JKL}{J_\text{\rm K-L}}
\newcommand{\Ex}{{\mathbf{E}}}
\newcommand{\Pb}{{\mathbf{P}}}

\begin{document}
\title{On misspecification in cusp-type change-point models}
\author[1]{O.V. Chernoyarov}
\author[2]{S. Dachian}
\author[3]{Yu.A. Kutoyants}
\affil[1,3]{\small National Research University ``MPEI'', Moscow, Russia,}
\affil[2]{\small University of Lille, Lille,  France,}
\affil[3]{\small Le Mans University,  Le Mans,  France}
\affil[1,3]{\small Tomsk State University, Tomsk, Russia}

\date{}
\maketitle

\begin{abstract}
The problem of parameter estimation by i.i.d.\ observations of an
inhomogeneous Poisson process is considered in situation of
misspecification. The model is that of a Poissonian signal observed in
presence of a homogeneous Poissonian noise.  The intensity function of the
process is supposed to have a cusp-type singularity at the change-point (the
unknown moment of arrival of the signal), while the supposed (theoretical) and
the real (observed) levels of the signal are different. The asymptotic
properties of pseudo MLE are described. It is shown that the estimator
converges to the value minimizing the Kullback-Leibler divergence, that the
normalized error of estimation converges to some limit distribution, and that
its polynomial moments also converge.

\end{abstract}


\section{Introduction}

It is commonplace in statistics that the theoretical models do not coincide
with the real models generating the observations. The properties of the
estimators constructed on the base of the theoretical models in such
situations do not coincide with their real properties. Sometimes this
difference between models can be important, and this requires a special study.
The study of such situations in statistics was initiated in the work of
Huber~\cite{Hub67}. There is a large diversity of publications devoted to
different statistical models and different types of misspecification. Special
attention is paid to the case when the real models are close to the
theoretical models. A nice theory of robust estimation was developed in the
book of Huber~\cite{Hub04}. The large majority of publication is devoted to
regular statistical models, i.e., to the case when both the theoretical and
the real models are sufficiently smooth w.r.t.\ the unknown parameter.  The
different cases of misspecification for change-point models (with jump-type
changes) were considered as well (see, e.g., the
publications~\cite{CKT18,DK01,Pi04,SBK16,YH90} and the references
therein). Note that in the works~\cite{CKT18} and~\cite{SBK16}, the
regularities of the theoretical and of the real models are different, e.g., it
is supposed that the model is of change-point type, while the real model of
observation is regular. It is known that the maximum likelihood estimator
(MLE) under misspecification converges to the value which minimizes the
Kulback-Leibler divergence between the measure corresponding to the
observations and the parametric family of theoretical measures. Usually this
value does not coincide with the true value, but in the change-point case,
there exists a large class of models admitting consistent estimation even
under misspecification (see, e.g.,~\cite{CKLG00,DK01}).

A recently introduced class of models (\emph{cups-type change-point\/} models,
or models with \emph{cusp-type singularity\/}) can be considered as an
intermediate between regular (smooth) and change-point (discontinuous) models
in the problem of estimation of the moment of arrival of a signal
(see~\cite{CDK18,CDK20,DKKN18}).  In the case when the observations are
inhomogeneous Poisson processes, the front of the arrival of the signal in
such models corresponds to a strongly increasing continuous intensity function
with infinite Fisher information. The examples of intensity functions in
regular~(a), cusp-type change-point~(b) and change-point~(c) models are given
in Fig.~\ref{F1}.

\begin{figure}[H]
\begin{center}
\includegraphics[width=0.7\textwidth]{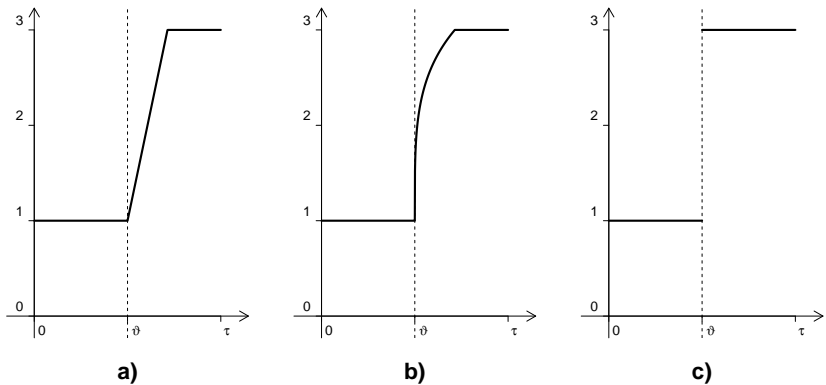}
\caption{Intensities with three types of fronts of arrival of a signal}
\label{F1}
\end{center}
\end{figure}

The mean-squared errors of the MLE $\hat\vartheta _n$ of the location
parameter $\vartheta$ by $n$ independent observations of a Poisson process
with these three types of fronts are
\[
\text{a)}\ \Ex_\vartheta\bigl(\hat\vartheta_n-\vartheta\bigr)^2 \approx
\frac{c}{n}\, ,\qquad
\text{b)}\ \Ex_\vartheta\bigl(\hat\vartheta_n-\vartheta\bigr)^2 \approx
\frac{c}{n^\gamma}\, ,\qquad
\text{c)}\ \Ex_\vartheta\bigl(\hat\vartheta_n-\vartheta\bigr)^2 \approx
\frac{c}{n^2}\, ,
\]
where $1<\gamma <2$ and $c$ are some constants (see, e.g.,~\cite{Kut22}). That
is why the cusp-type change-point models are considered as intermediate
between regular and change-point models.

It is reasonable to suppose that the real signals are continuous functions,
and that the cusp-type change-point models can provide a better fit of the
mathematical model to the real signals. An example of a discontinuous and a
close to it cusp-type intensity functions is given in Fig.~\ref{F2}.

\begin{figure}[H]
\begin{center}
\includegraphics[width=0.5\textwidth]{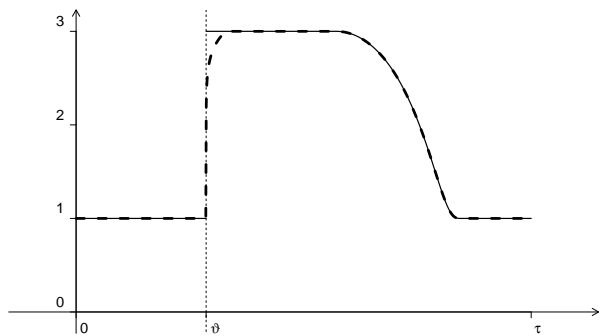}
\caption{Discontinuous (solid line) and cusp-type (dashed line) intensity
  functions}
\label{F2}
\end{center}
\end{figure}

Note that the statistical model of i.i.d.\ observations of a random variable
having a cusp-type singularity was first studied in~\cite{PR68}. Afterwards,
the parameter estimation problems for models with cusp-type singularities were
studied by many authors. For inhomogeneous Poisson processes this was done
in~\cite{D03,D11}, for diffusions with small noise in~\cite{Kut18}, for
ergodic diffusions in~\cite{DK03,Fu10}. Nonparametric estimation of a signal
with cusp-type front was considered in~\cite{Ra98}.

In this work we consider the problem of estimation of the time of arrival of a
signal having cusp-type singularity in the situation of misspecification of
the level of the signal. An example of intensity functions corresponding to
two signals (theoretical and real) with two different levels is given in
Fig.~\ref{F4}.

\begin{figure}[H]
\begin{center}
\includegraphics[width=0.5\textwidth]{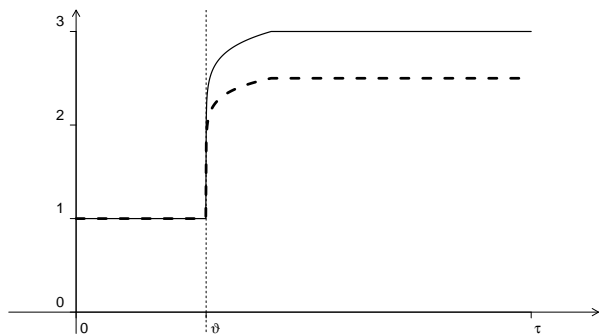}
\caption{Example of the  theoretical (dashed line)  and real  intensities.}
\label{F4}
\end{center}
\end{figure}

The main result of the paper are the asymptotic properties of the
(pseudo)~MLE.

\section{Asymptotic behavior of the pseudo MLE}

Suppose that a statistician has $n$ independent observations
$X^{(n)}=(X_1,\ldots,X_n)$, where $X_j = \bigl(X_j(t),\ 0\leq t\leq
\tau\bigr)$, for each $j=1,\ldots,n$, is an inhomogeneous Poisson process with
intensity function $\lambda_*(\vartheta,\cdot) = \bigl(\lambda_*(\vartheta
,t),\ 0\leq t\leq \tau\bigr) $, $\vartheta\in\Theta$. However, the function
$\lambda_*(\vartheta,\cdot)$ is unknown to the statistician and he uses a
different model with an intensity function $\lambda(\vartheta,\cdot) =
\bigl(\lambda(\vartheta,t),\ 0\leq t\leq \tau\bigr)$, $\vartheta\in\Theta$.

The unknown parameter $\vartheta$ is the time of arrival of a signal, the
latter being observed in presence of a homogeneous Poissonian noise of
intensity~$\lambda_0$. We consider the case when the front of the signal has a
cusp-type singularity, i.e., the statistician supposes that the intensity
function (called \emph{theoretical\/}) of the observed Poisson processes is
\begin{equation}
\label{1}
\lambda(\vartheta,t) = S\, \psi(t-\vartheta) + \lambda_0,\qquad
t\in[0,\tau],\ \vartheta\in\Theta,
\end{equation}
where the front of the signal $S\, \psi(t-\vartheta) $ is defined by the
function
\begin{equation}
\psi(x)=\Bigl(\frac{x}{\delta}\Bigr)^\kappa \1_{\{0<x<\delta\}} +
\1_{\{x\geq\delta\}},\qquad x\in\RR.
\end{equation}
The parameters $S>0$, $\lambda _0>0$ and $\kappa\in(0,1/2)$ are
supposed to be known.

However, the \emph{real\/} intensity function of the observed processes is
\begin{align}
\label{3}
\lambda_*(\vartheta_0,t) = (S+h)\, \psi(t-\vartheta_0) + \lambda_0,\qquad
t\in[0,\tau],\ \vartheta_0\in\Theta_0,
\end{align}
where $h$ is the \emph{contamination\/} of the signal.  Here we use different
sets $\Theta_0=(\alpha,\beta)$ and $\Theta=(\alpha-\delta,\beta+\delta)$ with
$\Theta_0\subset\Theta\subset(0,\tau-\delta)$.  The reason to consider in the
theoretical model a set $\Theta$ which is wider than the set $\Theta_0$ of
possible values of the parameter $\vartheta_0$ will become clear later.

The \emph{pseudo likelihood ratio\/} (p-LR) used by the statistician (see,
e.g.,~\cite{LS01}) is
\[
L\bigl(\vartheta, X^{(n)}\bigr) = \exp\Biggl\{\sum_{j=1}^{n} \int_{\vartheta}^{\tau}
\ln\bigl(\lambda(\vartheta,t)\bigr)\, \mathrm{d}X_j(t) - n \int_{\vartheta}^{\tau}
          [\lambda(\vartheta,t)-1]\, \mathrm{d}t\Biggr\}, \qquad
          \vartheta\in\Theta.
\]
We use the word ``pseudo'' since the intensity of the processes $X_j$,
$j=1,\ldots,n$, is not~$\lambda(\vartheta_0,\cdot)$, but
$\lambda_*(\vartheta_0,\cdot)$.

The \emph{pseudo maximum likelihood estimator\/} (p-MLE) $\hat\vartheta_n$ is
defined by the equation
\[
L\bigl(\hat\vartheta_n,X^{(n)}\bigr) = \sup_{\vartheta\in\Theta}
L\bigl(\vartheta, X^{(n)}\bigr).
\]

As it is usually the case in misspecified problems, the limit of the p-MLE
$\hat\vartheta_n$ will be given by the value~$\hat\vartheta $ which minimizes
the Kullback-Leibler divergence
\[
\JKL(\vartheta)=\int_{\vartheta\wedge\vartheta_0}^{\tau}
\biggl[\frac{\lambda(\vartheta,t)}{\lambda_*(\vartheta_0,t)} - 1 -
  \ln\biggl(\frac{\lambda(\vartheta,t)}{\lambda_*(\vartheta_0,t)}\biggr)\biggr]
\lambda_*(\vartheta_0,t)\, \mathrm{d}t, \qquad \vartheta\in\Theta.
\]

We need to introduce several notations. First, we introduce the random process
\[
\hat Z(u) = \exp\biggl(W^H(u)-\frac{u^2}{2}\biggr),\qquad u\in \RR,
\]
where $H=\kappa +1/2$, and $W^H(\cdot)$ is a two-sided fractional Brownian
motion (fBm) with Hurst parameter~$H$, i.e., a centered Gaussian process with
covariance
\[
\Ex \Bigl(W^H(u_1)W^H(u_2)\Bigr) = \frac{1}{2} \Bigl[\bil|u_1\bir|^{2H} +
  \bil|u_2\bir|^{2H} - \bil|u_1-u_2\bir|^{2H}\Bigr], \qquad u_1,u_2\in\RR.
\]
Let us recall here that $W^H$ admits the representations
\[
W^H(u) = \Gamma_\kappa^{-1} \int_{-\infty}^{+\infty} \bigl[(v-u)_+^\kappa -
  v_+^\kappa\bigr]\, \dd W(v) = \Gamma_\kappa^{-1} \int_{-\infty}^{+\infty}
\bigl[(u-s)_+^\kappa - (-s)_+^\kappa\bigr]\, \dd\widetilde W(s),
\]
where $W(\cdot)$ and $\widetilde W(\cdot)$ are two-sided Wiener processes
(Brownian motions).

Further, we introduce the random variable $\hat u_\kappa$ which is the (almost
surely unique) solution of the equation
\[
\hat Z(\hat u_\kappa) = \sup_{u\in \RR}\hat Z(u).
\]

Finally, we introduce the constants
\[
\Gamma _\kappa =\int_{\RR }^{}\bigl[(v-u)_+^\kappa - (v)_+^\kappa\bigr]^2\,
\dd v \quad\text{and}\quad
b=\left(\frac{S\,\Gamma_\kappa\,\sqrt{\lambda_*(\vartheta_0,\hat\vartheta)}}
{\lambda_0\,\delta^\kappa\,\JKL''(\hat\vartheta)}\right)^{\frac{2}{3-2\kappa}},
\]
as well as the set 
\begin{equation}
\label{4}
\mathcal{H}=\biggl\{h \;:\; h > \frac{S}{\ln\bigl(1+\frac{S}{\lambda_0}\bigr)}
- S - \lambda_0\biggr\}
\end{equation}

Note that the explicit expression of $\JKL''(\hat\vartheta)$ is given in
Proposition~\ref{P1} below, where it is equally shown that if the
contamination $h\in\mathcal{H}$, then the Kullback-Leibler divergence has a
unique minimum at some point $\hat\vartheta\in\Theta$.

Note also that the condition $h\in {\cal H}$  can be rewritten as
\begin{align*}
S+h+\lambda_0 >
\frac{(S+\lambda_0)-\lambda_0}{\ln(S+\lambda_0)-\ln(\lambda_0)}\, ,
\end{align*}
and so it coincides with the condition providing the consistent estimation in
a similar misspecified problem for the change-point case~\cite{Kut22}.

Note finally, that the right hand side of the inequality in~\eqref{4} belongs
to the interval $(-S,0)$ (this follows immediately from the elementary
inequalities $\ln(x) < x-1$ and $\ln(x) > 1-\frac{1}{x}$ for $x\ne 1$), and so
the contamination can be positive or negative.  For the case $\lambda_0=1$ (we
can always reduce ourselves to this case dividing $h$ and $S$ by $\lambda_0$),
the region of admissible values of $h$ (as function of $S$) is represented in
Fig.~\ref{h-region}.

\begin{figure}[H]
\begin{center}
\includegraphics[width=0.4\textwidth]{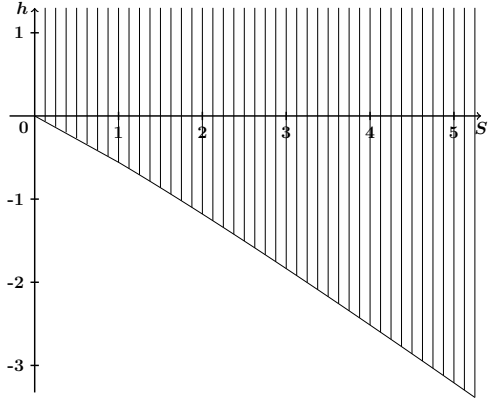}
\caption{Possible values of the contamination~$h$}
\label{h-region}
\end{center}
\end{figure}

Now we can state the main result of the present paper.

\begin{theorem}
\label{T1}
Suppose that we have the model\/~\textup{\eqref{1}--\eqref{3}} and\/
$h\in\mathcal{H}\setminus\{0\}$.  Then the p-MLE\/~$\hat\vartheta _n$ is
``consistent'':\/ $\hat\vartheta_n \stackrel{\Pb}{\longrightarrow}
\hat\vartheta $, converges in distribution:
\begin{equation}
\label{mle}
n^{\frac{1}{3-2\kappa}}\; b^{-1}\, \bigl(\hat\vartheta_n - \hat\vartheta\bigr)
\Longrightarrow \hat u_\kappa,
\end{equation}
and we have the convergence of polynomial moments: for any\/ $p>0$, it holds
\[
n^{\frac{p}{3-2\kappa }}\; \Ex \bigl|\hat\vartheta_n - \hat\vartheta\bigr|^p
\longrightarrow b^p\, \Ex \bigl|\hat u_\kappa\bigr|^p.
\]
\end{theorem}

\begin{proof}
To prove this theorem we use the approach developed by Ibragimov and
Khasminskii in~\cite{IH81}, which is based on the weak convergence of the
normalized likelihood ratio process to some limit process. The particularity
of the misspecified models concerns the study of the corresponding random
functions, which are not true likelihood ratios, and hence the direct
application of Theorem 1.10.1 of~\cite{IH81} is often impossible. In our case,
we follow the modification of this method introduced in~\cite{CKT18,CDK18}
(see as well~\cite{Kut22}). Let us explain how it works.

We put
\[
\varphi_n = b n^{-\frac{1}{3-2\kappa}} \longrightarrow 0
\]
and introduce the normalized p-LR
\[
Z_n(u)=\frac{L\bigl(\hat\vartheta+\varphi_n u,
  X^{(n)}\bigr)}{L\bigl(\hat\vartheta, X^{(n)}\bigr)},\qquad u\in\UU_n =
\biggl(\frac{\alpha-\delta-\hat\vartheta}{\varphi_n},
\frac{\beta+\delta-\hat\vartheta}{\varphi_n}\biggr).
\]

Note that since $\hat\vartheta\in\Theta$, we have $\UU_n\uparrow\RR$.  For any
fixed $u\not=0$, we have $Z_n(u) \longrightarrow \infty$, that is why we
introduce a second normalization: we put
\[
\varepsilon_n= \frac{1}{b^2\JKL''(\hat\vartheta)}\,
n^{-\frac{1-2\kappa}{3-2\kappa}} \longrightarrow 0 \quad\text{and}\quad \hat
Z_n(u)=\bigl[Z_n(u)\bigr]^{\varepsilon_n}
\]

Then we show the following three lemmas (the proofs are in the next section).

\begin{lemma}
\label{L1}
Under the hypotheses of Theorem~\ref{T1}, the finite dimensional distributions
of the process\/~$\hat Z_n(\cdot)$ converge to those of the process\/~$\hat
Z(\cdot)$.
\end{lemma}

\begin{lemma}
\label{L2}
Under the hypotheses of Theorem~\ref{T1}, there exist some constants\/ $c,C>0$
such that
\begin{equation}
\label{ineqL2}
\Ex \hat Z_n^{1/2}(u) \leq C\exp\bigl\{-c u^2 \bigr\}
\end{equation}
for all $n\in\NN$ and $u\in\UU_n$.
\end{lemma}

\begin{lemma}
\label{L3}
Under the hypotheses of Theorem~\ref{T1}, there exist some constants\/ $C>0$
and\/ $\gamma>1$ such that
\[
\Ex \bigl[\hat Z_n^{1/2}(u_1) - \hat Z_n^{1/2}(u_2)\bigr]^2 \leq C\, \bil|u_1 -
u_2\bir|^\gamma
\]
for all\/ $n\in\NN$ and\/ $u_1,u_2\in\UU_n$.
\end{lemma}

The properties of the normalized p-LR $ \hat Z_n(\cdot)$ established in the
Lemmas \ref{L1}-\ref{L3} allow us to apply a modification of Theorem 1.10.1 in
\cite{IH81} and obtain all the desired properties of the p-MLE.  For example,
the convergence~\eqref{mle} is proved as follows. Using the change of variable
$\vartheta = \hat\vartheta +\varphi_nu$, we can write
\begin{align*}
\Pb\Biggl(\frac{\hat\vartheta_n-\hat\vartheta }{\varphi_n} < x\Biggr) &=
\Pb\bigl(\hat\vartheta_n < \hat\vartheta+\varphi_n x\bigr) =
\Pb\Biggl(\sup_{\vartheta < \hat\vartheta+\varphi_n x}
L\bigl(\vartheta,X^{(n)}\bigr) > \sup_{\vartheta \geq \hat\vartheta+\varphi_n
  x} L\bigl(\vartheta,X^{(n)}\bigr)\Biggr)\\
&= \Pb\Biggl(\sup_{\vartheta < \hat\vartheta+\varphi_n x}
\frac{L\bigl(\vartheta,X^{(n)}\bigr)}{L\bigl(\hat\vartheta,X^{(n)}\bigr)} >
\sup_{\vartheta \geq \hat\vartheta+\varphi_n x}
\frac{L\bigl(\vartheta,X^{(n)}\bigr)}{L\bigl(\hat\vartheta,X^{(n)}\bigr)}\Biggr)\\
&= \Pb\biggl(\sup_{u < x} Z_n(u) > \sup_{u\geq x} Z_n(u)\biggr) =
\Pb\biggl(\sup_{u < x} \hat Z_n(u) > \sup_{u\geq x} \hat Z_n(u)\biggr)\\
&\longrightarrow \Pb\biggl(\sup_{u < x} \hat Z(u) > \sup_{u\geq x} \hat
Z(u)\biggr) = \Pb\bigl(\hat u_\kappa < x\bigr),
\end{align*}
as soon as
\[
\Pb\biggl(\sup_{u < x} \hat Z(u) = \sup_{u\geq x} \hat Z(u)\biggr) = 0,
\]
that is, as soon as $\Pb\bigl(\hat u_\kappa = x\bigr)=0$.
\end{proof}

\section{Proof of the lemmas}

We start with the following proposition which gives some properties of the
Kullback-Leibler divergence $\JKL(\cdot)$ in our case.

\begin{proposition}
\label{P1}
Suppose that\/ $h\in\mathcal{H}$. Then the function\/ $\JKL(\cdot)$ has the
following properties.
\begin{enumerate}[label={${\it\roman*})$}]
\item\label{L0i} The function\/ $\JKL(\cdot)$ is continuously differentiable,
  and
\[
\JKL'(\vartheta) = \begin{cases}
\vphantom{\Big(}\lambda_0\ln\bigl(1+\frac{S}{\lambda_0}\bigr) - S, &
\text{if\/~} \vartheta\leq\vartheta_0-\delta,\\
\vphantom{\bigg(}\lambda_0 \ln\Bigl(1+\frac{S}{\lambda_0}
\bigl(\frac{\vartheta_0-\vartheta}{\delta}\bigr)^\kappa\Bigr) - S
\bigl(\frac{\vartheta_0-\vartheta}{\delta }\bigr)^\kappa + I_1(\vartheta), &
\text{if\/~} \vartheta\in[\vartheta_0-\delta,\vartheta_0],\\
\vphantom{\bigg(}\lambda_+ \ln\Bigl(\frac{S+\lambda_0}{S
  \bigl(1-\frac{\vartheta_0-\vartheta}{\delta }\bigr)^\kappa+\lambda_0}\Bigr)
+ S \bigl(1-\frac{\vartheta_0-\vartheta}{\delta }\bigr)^\kappa -S +
I_2(\vartheta), & \text{if\/~} \vartheta\in[\vartheta_0,\vartheta_0+\delta],\\
\vphantom{\Big(}\lambda_+ \ln\bigl(1+\frac{S}{\lambda_0}\bigr) - S, &
\text{\/if~}\vartheta\geq\vartheta_0+\delta,
\end{cases}
\]
where\/ $\lambda_+=S+h+\lambda_0$ and
\begin{align*}
I_1(\vartheta) &= \int_{\frac{\vartheta_0-\vartheta}{\delta}}^{1} S \kappa x^{\kappa-1}\,
\frac{(S+h) \bigl(x-\frac{\vartheta_0-\vartheta}{\delta}\bigr)^\kappa - S
  x^\kappa}{S x^\kappa + \lambda_0}\;\mathrm{d}x,\\
I_2(\vartheta) &= \int_{0}^{1-\frac{\vartheta-\vartheta_0}{\delta}} S \kappa x^{\kappa-1}\,
\frac{(S+h) \bigl(x+\frac{\vartheta-\vartheta_0}{\delta}\bigr)^\kappa - S
  x^\kappa}{S x^\kappa + \lambda_0}\;\mathrm{d}x.
\end{align*}

In particular,
\begin{itemize}[label={--}]
\item $\JKL'(\vartheta) < 0$ for\/
  $\vartheta\leq\vartheta_0-\delta$,
\item $\JKL'(\vartheta) > 0$ for\/
  $\vartheta\geq\vartheta_0+\delta$,
\item $\JKL'(\vartheta_0) = Ah$ with\/
  $A=1-\frac{\lambda_0}{S}\ln\bigl(1+\frac{S}{\lambda_0}\bigr) > 0$.
\end{itemize}
\item\label{L0ii} The function\/ $\JKL(\cdot)$ is twice continuously
  differentiable everywhere except at the point\/~$\vartheta_0$ (in which\/
  $\JKL''(\vartheta_0)=+\infty$), and
\[
\JKL''(\vartheta) = \begin{cases}
\displaystyle\vphantom{\int_\int^\int}\frac{S (S+h) \kappa^2}{\delta}
\int_{\frac{\vartheta_0-\vartheta}{\delta}}^{1} \frac{x^{\kappa-1}
  \bigl(x-\frac{\vartheta_0-\vartheta}{\delta}\bigr)^{\kappa-1}}{S x^\kappa +
  \lambda_0}\; \mathrm{d}x, & \text{if\/~}
\vartheta\in[\vartheta_0-\delta,\vartheta_0),\\
\displaystyle\vphantom{\int_\int^\int}\frac{S (S+h) \kappa^2}{\delta}
\int_{0}^{1-\frac{\vartheta-\vartheta_0}{\delta}} \frac{x^{\kappa-1}
  \bigl(x+\frac{\vartheta-\vartheta_0}{\delta}\bigr)^{\kappa-1}}{S x^\kappa +
  \lambda_0}\; \mathrm{d}x, & \text{\/if~}
\vartheta\in(\vartheta_0,\vartheta_0+\delta],\\
0, & \text{\/if~} \vartheta\notin(\vartheta_0-\delta,\vartheta_0+\delta).\\
\end{cases}
\]

In particular, $\JKL''(\cdot)>0$ on\/
$(\vartheta_0-\delta,\vartheta_0)\cup(\vartheta_0,\vartheta_0+\delta)$, and
hence the function\/ $\JKL'(\cdot)$ is strictly increasing on\/
$(\vartheta_0-\delta,\vartheta_0+\delta)$.
\item\label{L0iii} The function\/ $\JKL(\cdot)$ attains its unique minimum at
  the point\/ $\hat\vartheta$ which is the (unique) solution of the equation\/
  $\JKL'(\hat\vartheta)=0$.  Moreover, if\/ $h>0$, we have\/
  $\hat\vartheta\in(\vartheta_0-\delta,\vartheta_0)$, and if\/ $h<0$, we
  have\/ $\hat\vartheta\in(\vartheta_0,\vartheta_0+\delta)$ (of course\/
  $\hat\vartheta=\vartheta_0$ for $h=0$).
\item\label{L0iv} If\/ $h\ne 0$, there exist some constants\/ $m,M>0$ such that we
  have (for all\/ $\vartheta\in\Theta$) the estimates
\begin{equation}
\label{borneJ'}
m \bil|\vartheta-\hat\vartheta\bir| \leq \bil|\JKL'(\vartheta)\bir| \leq M
\bil|\vartheta-\hat\vartheta\bir|,
\end{equation}
and consequently
\begin{equation}
\label{borneJ}
\frac{m}{2}\, (\vartheta-\hat\vartheta)^2 \leq
\JKL(\vartheta)-\JKL(\hat\vartheta) \leq \frac{M}{2}\,
(\vartheta-\hat\vartheta)^2.
\end{equation}
\end{enumerate}
\end{proposition}

\begin{proof}
Throughout the proof, $C$ denotes a ``generic'' quantity not depending
on~$\vartheta$, which can vary from formula to formula (and even in the same
formula).  Note also that since we have supposed $\vartheta_0\in\Theta_0 =
(\alpha,\beta)$, we have $\vartheta_0-\delta,\vartheta_0+\delta \in \Theta =
(\alpha-\delta,\beta+\delta)$.

For $\vartheta\leq\vartheta_0-\delta$, we can write
\begin{align*}
\JKL(\vartheta) &= \int_{\vartheta}^{\vartheta+\delta}
\biggl[\frac{S \bigl(\frac{t-\vartheta}{\delta}\bigr)^\kappa +
    \lambda_0}{\lambda_0} - 1 - \ln\biggl(\frac{S
    \bigl(\frac{t-\vartheta}{\delta}\bigr)^\kappa +
    \lambda_0}{\lambda_0}\biggr)\biggr] \lambda_0\, \mathrm{d}t\\*
&\phantom{{}=} + \int_{\vartheta+\delta}^{\vartheta_0} \biggl[\frac{S +
    \lambda_0}{\lambda_0} - 1 - \ln\biggl(\frac{S +
    \lambda_0}{\lambda_0}\biggr)\biggr] \lambda_0\, \mathrm{d}t + C\\
&= \delta \int_{0}^{1} \biggl[S x^\kappa - \lambda_0\ln
  \biggl(1+\frac{S}{\lambda_0}\, x^\kappa\biggr)\biggr] \mathrm{d}x +
(\vartheta_0-\vartheta-\delta) \biggl[S - \lambda_0
  \ln\biggl(1+\frac{S}{\lambda_0}\biggr)\biggr] + C\\*
&= \vartheta\biggl[\lambda_0 \ln\biggl(1+\frac{S}{\lambda_0}\biggr) -
  S\biggr] + C.
\end{align*}
Hence, in this case, $\JKL'(\vartheta) =
\lambda_0\ln\bigl(1+\frac{S}{\lambda_0}\bigr) - S$ and $\JKL''(\vartheta) =
0$.  The fact that $\JKL'(\vartheta)<0$ follows immediately from the
elementary inequality $\ln(x) < x-1$ for~$x\ne 0$.

Similarly, for $\vartheta\geq\vartheta_0+\delta$, we have
\begin{align*}
\JKL(\vartheta) &= \int_{\vartheta_0+\delta}^{\vartheta}
\biggl[\frac{\lambda_0}{\lambda_+} - 1 -
  \ln\biggl(\frac{\lambda_0}{\lambda_+}\biggr)\biggr] \lambda_+\,
\mathrm{d}t\\*
&\phantom{{}=} + \int_{\vartheta}^{\vartheta+\delta} \biggl[\frac{S
    \bigl(\frac{t-\vartheta}{\delta}\bigr)^\kappa + \lambda_0}{\lambda_+} - 1
  - \ln\biggl(\frac{S \bigl(\frac{t-\vartheta}{\delta}\bigr)^\kappa +
    \lambda_0}{\lambda_+}\biggr)\biggr] \lambda_+\, \mathrm{d}t\\*
&\phantom{{}=} + \int_{\vartheta+\delta}^{\tau} \biggl[\frac{S +
    \lambda_0}{\lambda_+} - 1 - \ln\biggl(\frac{S +
    \lambda_0}{\lambda_+}\biggr)\biggr] \lambda_+\, \mathrm{d}t + C\\
&= (\vartheta-\vartheta_0-\delta) \biggl[-S - h - \lambda_+
  \ln\biggl(\frac{\lambda_0}{\lambda_+}\biggr)\biggr]\\*
&\phantom{{}=} + \delta \int_{0}^{1} \biggl[S x^\kappa - S - h - \lambda_+
  \ln\biggl(\frac{S x^\kappa +\lambda_0}{\lambda_+}\biggr)\biggr]
\mathrm{d}x\\*
&\phantom{{}=} + (\tau-\vartheta-\delta) \biggl[-h - \lambda_+
  \ln\biggl(\frac{S+\lambda_0}{\lambda_+}\biggr)\biggr] + C\\*
&= \vartheta\biggl[\lambda_+ \ln\biggl(1+\frac{S}{\lambda_0}\biggr) - S\biggr]
+ C.
\end{align*}
Hence, in this case, $\JKL'(\vartheta) =
\lambda_+\ln\bigl(1+\frac{S}{\lambda_0}\bigr) - S$ and $\JKL''(\vartheta) =
0$.  The fact that $\JKL'(\vartheta)>0$ follows immediately from the condition
$h\in\mathcal{H}$.

Now, in the case $\vartheta\in[\vartheta_0-\delta,\vartheta_0]$, denoting for
shortness $\phi(y)=(S+h) y^\kappa + \lambda_0$, it comes
\begin{align*}
\JKL(\vartheta) &= \int_{\vartheta}^{\vartheta_0} \biggl[\frac{S
    \bigl(\frac{t-\vartheta}{\delta}\bigr)^\kappa + \lambda_0}{\lambda_0} - 1
  - \ln\biggl(\frac{S \bigl(\frac{t-\vartheta}{\delta}\bigr)^\kappa +
    \lambda_0}{\lambda_0}\biggr)\biggr] \lambda_0\, \mathrm{d}t\\*
&\phantom{{}=} + \int_{\vartheta_0}^{\vartheta+\delta} \biggl[\frac{S
    \bigl(\frac{t-\vartheta}{\delta}\bigr)^\kappa +
    \lambda_0}{\phi\bigl(\frac{t-\vartheta_0}{\delta}\bigr)} - 1 -
  \ln\biggl(\frac{S \bigl(\frac{t-\vartheta}{\delta}\bigr)^\kappa +
    \lambda_0}{\phi\bigl(\frac{t-\vartheta_0}{\delta}\bigr)}\biggr)\biggr]
\phi\biggl(\frac{t-\vartheta_0}{\delta}\biggr)\, \mathrm{d}t\\*
&\phantom{{}=} + \int_{\vartheta+\delta}^{\vartheta_0+\delta} \biggl[\frac{S +
    \lambda_0}{\phi\bigl(\frac{t-\vartheta_0}{\delta}\bigr)} - 1 -
  \ln\biggl(\frac{S +
    \lambda_0}{\phi\bigl(\frac{t-\vartheta_0}{\delta}\bigr)}\biggr)\biggr]
\phi\biggl(\frac{t-\vartheta_0}{\delta}\biggr)\, \mathrm{d}t + C\\
&= \delta \int_{0}^{\frac{\vartheta_0-\vartheta}{\delta}} \biggl[S x^\kappa -
  \lambda_0 \ln\biggl(1+\frac{S}{\lambda_0}\, x^\kappa\biggr)\biggr]
\mathrm{d}x\\*
&\phantom{{}=} + \delta \int_{0}^{1-\frac{\vartheta_0-\vartheta}{\delta}}
\biggl[S \biggl(y+\frac{\vartheta_0-\vartheta}{\delta}\biggr)^\kappa - (S+h)
  y^\kappa - \phi(y) \ln\biggl(\frac{S
    \bigl(y+\frac{\vartheta_0-\vartheta}{\delta}\bigr)^\kappa +
    \lambda_0}{\phi(y)}\biggr)\biggr] \mathrm{d}y\\*
&\phantom{{}=} + \delta \int_{1-\frac{\vartheta_0-\vartheta}{\delta}}^{1}
\biggl[S - (S+h) y^\kappa - \phi(y) \ln\biggl(\frac{S +
    \lambda_0}{\phi(y)}\biggr)\biggr] \mathrm{d}y + C.
\end{align*}
Therefore, differentiating with respect to $\vartheta$, we get in this case
\begin{align*}
\JKL'(\vartheta) &= -\biggl[S
  \biggl(\frac{\vartheta_0-\vartheta}{\delta}\biggr)^\kappa - \lambda_0
  \ln\biggl(1+\frac{S}{\lambda_0}
  \biggl(\frac{\vartheta_0-\vartheta}{\delta}\biggr)^\kappa\biggr)\biggr]\\*
&\phantom{{}=} + \biggl[S - (S+h)
  \biggl(1-\frac{\vartheta_0-\vartheta}{\delta}\biggr)^\kappa -
  \phi\biggl(1-\frac{\vartheta_0-\vartheta}{\delta}\biggr) \ln\biggl(\frac{S +
    \lambda_0}{\phi\bigl(1-\frac{\vartheta_0-\vartheta}{\delta}\bigr)}\biggr)\biggr]\\*
&\phantom{{}=} + \delta \int_{0}^{1-\frac{\vartheta_0-\vartheta}{\delta}}
\biggl[-\frac{S \kappa}{\delta}
  \biggl(y+\frac{\vartheta_0-\vartheta}{\delta}\biggr)^{\kappa-1} + \phi(y)\,
  \frac{\frac{S \kappa}{\delta}
    \bigl(y+\frac{\vartheta_0-\vartheta}{\delta}\bigr)^{\kappa-1}} {S
    \bigl(y+\frac{\vartheta_0-\vartheta}{\delta}\bigr)^\kappa +
    \lambda_0}\biggr] \mathrm{d}y\\*
&\phantom{{}=} - \biggl[S - (S+h)
  \biggl(1-\frac{\vartheta_0-\vartheta}{\delta}\biggr)^\kappa -
  \phi\biggl(1-\frac{\vartheta_0-\vartheta}{\delta}\biggr) \ln\biggl(\frac{S +
    \lambda_0}{\phi\bigl(1-\frac{\vartheta_0-\vartheta}{\delta}\bigr)}\biggr)\biggr]\\
&= \lambda_0 \ln\biggl(1+\frac{S}{\lambda_0}
\biggl(\frac{\vartheta_0-\vartheta}{\delta}\biggr)^\kappa\biggr) - S
\biggl(\frac{\vartheta_0-\vartheta}{\delta}\biggr)^\kappa\\*
&\phantom{{}=} + \int_{\frac{\vartheta_0-\vartheta}{\delta}}^{1} \biggl[-S
  \kappa x^{\kappa-1} +
  \phi\biggl(x-\frac{\vartheta_0-\vartheta}{\delta}\biggr)\, \frac{S \kappa
    x^{\kappa-1}} {S x^\kappa + \lambda_0}\biggr] \mathrm{d}x\\
&= \lambda_0 \ln\biggl(1+\frac{S}{\lambda_0}
\biggl(\frac{\vartheta_0-\vartheta}{\delta}\biggr)^\kappa\biggr) - S
\biggl(\frac{\vartheta_0-\vartheta}{\delta}\biggr)^\kappa\\*
&\phantom{{}=} + \int_{\frac{\vartheta_0-\vartheta}{\delta}}^{1} S \kappa
x^{\kappa-1}\, \frac{(S+h)
  \bigl(x-\frac{\vartheta_0-\vartheta}{\delta}\bigr)^\kappa - S x^\kappa} {S
  x^\kappa + \lambda_0}\; \mathrm{d}x\\
&= \lambda_0 \ln\biggl(1+\frac{S}{\lambda_0}
\biggl(\frac{\vartheta_0-\vartheta}{\delta}\biggr)^\kappa\biggr) - S
\biggl(\frac{\vartheta_0-\vartheta}{\delta}\biggr)^\kappa + I_1(\vartheta).
\end{align*}

For $\vartheta\in[\vartheta_0-\delta,\vartheta_0)$, differentiating once more
  with respect to $\vartheta$, we obtain
\begin{align*}
\JKL''(\vartheta) &= \frac{-\frac{S \kappa}{\delta}
  \bigl(\frac{\vartheta_0-\vartheta}{\delta}\bigr)^{\kappa-1}
}{1+\frac{S}{\lambda_0}
  \bigl(\frac{\vartheta_0-\vartheta}{\delta}\bigr)^\kappa} + \frac{S
  \kappa}{\delta}
\biggl(\frac{\vartheta_0-\vartheta}{\delta}\biggr)^{\kappa-1} +
I_1'(\vartheta)\\*
 &= \frac{S \kappa}{\delta}
\biggl(\frac{\vartheta_0-\vartheta}{\delta}\biggr)^{\kappa-1} \biggl[1 -
  \frac{1}{1+\frac{S}{\lambda_0}
    \bigl(\frac{\vartheta_0-\vartheta}{\delta}\bigr)^\kappa}\biggr] + \frac{S
  \kappa}{\delta}
\biggl(\frac{\vartheta_0-\vartheta}{\delta}\biggr)^{\kappa-1}\, \frac{- S
  \bigl(\frac{\vartheta_0-\vartheta}{\delta}\bigr)^\kappa} {S
  \bigl(\frac{\vartheta_0-\vartheta}{\delta}\bigr)^\kappa + \lambda_0}\\*
&\phantom{{}=} + \int_{\frac{\vartheta_0-\vartheta}{\delta}}^{1} S \kappa
x^{\kappa-1}\, \frac{\frac{(S+h)\kappa}{\delta}
  \bigl(x-\frac{\vartheta_0-\vartheta}{\delta}\bigr)^{\kappa-1}} {S x^\kappa +
  \lambda_0}\; \mathrm{d}x\\*
&= \frac{S(S+h)\kappa^2}{\delta}
\int_{\frac{\vartheta_0-\vartheta}{\delta}}^{1} \frac{x^{\kappa-1}
  \bigl(x-\frac{\vartheta_0-\vartheta}{\delta}\bigr)^{\kappa-1}} {S x^\kappa +
  \lambda_0}\; \mathrm{d}x.
\end{align*}
Note that for $\vartheta\in(\vartheta_0-\delta,\vartheta_0)$ the last integral
is strictly positive, and that for $\vartheta=\vartheta_0$ it would become
$\int_{0}^{1} \frac{x^{2\kappa-2}} {S x^\kappa + \lambda_0}\, \mathrm{d}x$ and
diverge to $+\infty$ (since $2\kappa-2<-1$).

The calculation of $\JKL'(\vartheta)$ and $\JKL''(\vartheta)$ in the remaining
case $\vartheta\in[\vartheta_0,\vartheta_0+\delta]$ can be carried out in a
similar way.

So, to conclude the proof of the parts~\ref{L0i} and~\ref{L0ii} of the lemma,
it remains to show that~$\JKL'(\vartheta_0)=Ah$.  Indeed,
\begin{align*}
\JKL'(\vartheta_0) &= I_1(\vartheta_0) = I_2(\vartheta_0) =
\int_{0}^{1} S \kappa x^{\kappa-1}\, \frac{h x^\kappa} {S x^\kappa +
  \lambda_0}\; \mathrm{d}x = \int_{0}^{1} \frac{S h y}{S y + \lambda_0}\;
\mathrm{d}y\\*
&= h \int_{0}^{1}
\biggl[1 - \frac{\lambda_0}{S y + \lambda_0}\biggr] \mathrm{d}y = h \biggl[y -
  \frac{\lambda_0}{S} \ln\biggl(1 + \frac{S}{\lambda_0} y\biggr)\biggr]_0^1 =
h \biggl[1 - \frac{\lambda_0}{S} \ln\biggl(1 +
  \frac{S}{\lambda_0}\biggr)\biggr].
\end{align*}

The parts~\ref{L0iii} of the lemma follows directly from the parts~\ref{L0i}
and~\ref{L0ii}.  So, it remains to prove the part~\ref{L0iv}.

As $h \ne 0$, we have $\hat\vartheta \ne \vartheta_0$, and hence
$\JKL''(\hat\vartheta) > 0$.  So, there exist some vicinity of $\hat\vartheta$
and some constants~$m,M>0$, such that we have $m < \JKL''(\vartheta) < M$ for
$\vartheta$ belonging to this vicinity.  Hence, as
\[
\bil|\JKL'(\vartheta)\bir| = \bil|\JKL'(\vartheta) -
\JKL'(\hat\vartheta)\bir| = \JKL''(\tilde\vartheta)
\bil|\vartheta-\hat\vartheta\bir|,
\]
where $\tilde\vartheta$ is some intermediate value between $\vartheta$ and
$\hat\vartheta$, the estimates~\eqref{borneJ'} are valid for $\vartheta$
belonging to this vicinity. Noting that the function $\JKL'(\cdot)$ is
non-decreasing and bounded, this inequalities can be clearly extended to the
whole~$\Theta$ by adjusting the constants $m$ and~$M$.

The estimates~\eqref{borneJ} follow easily from the
estimates~\eqref{borneJ'}. For example, the upper estimate in the case
$\vartheta<\hat\vartheta$ can be obtained as follows
\[
\JKL(\vartheta)-\JKL(\hat\vartheta) =
-\int_{\vartheta}^{\hat\vartheta} \JKL'(t)\, \mathrm{d}t =
\int_{\vartheta}^{\hat\vartheta} \bil|\JKL'(t)\bir|\, \mathrm{d}t
\leq M \int_{\vartheta}^{\hat\vartheta} (\hat\vartheta-t)\, \mathrm{d}t =
\frac{M}{2}\, (\hat\vartheta-\vartheta)^2.
\]
The proposition  is proved.
\end{proof}

Now we turn to the proof of the lemmas.

\begin{proof}[Proof of Lemma~\ref{L1}]
Let us  note that the theoretical intensity function can be rewritten as
\[
\lambda(\vartheta,t) = S\, \biggl(\frac{t-\vartheta}{\delta}\biggr)^\kappa
\1_{\{0<t-\vartheta<\delta\}} + S\, \1_{\{t-\vartheta \geq\delta\}} +
\lambda_0 = S\, \biggl(\frac{t-\vartheta}{\delta}\biggr)_+^\kappa +
\tilde\psi(t-\vartheta),
\]
where the function
\[
\tilde\psi(x) = S \Bigl[1 - \Bigl(\frac{x}{\delta}\Bigr)^\kappa\Bigr]
\1_{\{x\geq\delta\}} + \lambda_0,\qquad x\in\RR,
\]
is Lipschitz continuous:
\[
\bigl|\tilde\psi(x) - \tilde\psi(y)\bigr| \leq C \bil|x-y\bir|,\qquad
x,y\in\RR,
\]
with $C = \bigl|{\tilde\psi}'(\delta)\bigr| = \frac{S\kappa}{\delta}\,$.

Denoting $\vartheta_u = \hat\vartheta + \varphi_n u$ and 
\[
W_n(t) = \frac{1}{\sqrt{n}} \sum_{j=1}^{n}\Biggl[X_j(t) - \int_{0}^{t}
  \lambda_*(\vartheta_0,s)\, \dd s\Biggr],
\]
we can write
\begin{align*}
\ln\hat Z_n(u) &= \varepsilon_n \sum_{j=1}^{n} \int_{0}^{\tau}
\ln\biggl(\frac{\lambda(\vartheta_u,t)}{\lambda(\hat\vartheta,t)} \biggr)\,
\dd X_j(t) - n\varepsilon_n \int_{0}^{\tau} \bigl[\lambda(\vartheta_u,t) -
  \lambda(\hat\vartheta,t)\bigr]\, \dd t\\*
&= \varepsilon_n \sum_{j=1}^{n} \int_{0}^{\tau}
\ln\biggl(\frac{\lambda(\vartheta_u,t)}{\lambda(\hat\vartheta,t)} \biggr)
\bigl[\dd X_j(t) - \lambda_*(\vartheta_0,t) \dd t\bigr]\\*
&\phantom{={}} - n\varepsilon_n \int_{0}^{\tau} \biggl[\lambda(\vartheta_u,t)
  - \lambda(\hat\vartheta,t) - \lambda_*(\vartheta_0,t)
  \ln\biggl(\frac{\lambda(\vartheta_u,t)}{\lambda(\hat\vartheta,t)} \biggr)
  \biggr] \dd t\\
&= \varepsilon_n \sqrt{n} \int_{0}^{\tau}
\ln\biggl(\frac{\lambda(\vartheta_u,t)}{\lambda(\hat\vartheta,t)} \biggr)\,
\dd W_n(t) - n\varepsilon_n \bigl[\JKL(\vartheta_u) - \JKL(\hat\vartheta)\bigr]\\*
&= A_n(u) - B_n(u)
\end{align*}
with evident notations.

For $B_n(u)$, we have
\[
B_n(u) = n\varepsilon_n \bigl[\JKL(\hat\vartheta+u\varphi_n) -
  \JKL(\hat\vartheta)\bigr] = n\varepsilon_n\, \frac{\JKL''(\hat\vartheta)}2\,
(u\varphi_n)^2 + o(n\varepsilon_n\varphi_n^2) = \frac{u^2}{2} + o(1).
\]
Here we used the Taylor expansion of the function $\JKL$ in the vicinity of
the point $\hat\vartheta$, the fact that $\JKL'(\hat\vartheta) = 0$, and the
equality $n\varepsilon_n\varphi_n^2\JKL''(\hat\vartheta) = 1$.

For $A_n(u)$, using the Taylor expansion of the function
$x\longmapsto\ln(1+x)$, we get
\begin{align*}
A_n(u) &= \varepsilon_n \sqrt{n} \int_{0}^{\tau}
\ln\biggl(\frac{\lambda(\hat\vartheta+u\varphi_n,t)}{\lambda(\hat\vartheta,t)}
\biggr)\, \dd W_n(t)\\*
&= \varepsilon_n \sqrt{n} \int_{0}^{\tau}
\frac{\lambda(\hat\vartheta+u\varphi_n,t) -
  \lambda(\hat\vartheta,t)}{\lambda(\hat\vartheta,t)}\,\, \dd W_n(t)\,
\bigl(1+o_\Pb(1)\bigr)\\
&= \varepsilon_n \sqrt{n} \int_{0}^{\tau}
\frac{S\,\bigl(\frac{t-\hat\vartheta-u\varphi_n}{\delta}\bigr)_+^\kappa -
  S\,\bigl(\frac{t-\hat\vartheta}{\delta}\bigr)_+^\kappa}{\lambda(\hat\vartheta,t)}\,\,
\dd W_n(t)\, \bigl(1+o_\Pb(1)\bigr)\\
&\phantom{={}} + \varepsilon_n \sqrt{n} \int_{0}^{\tau}
\frac{\tilde\psi(t-\hat\vartheta-u\varphi_n) -
  \tilde\psi(t-\hat\vartheta)}{\lambda(\hat\vartheta,t)}\,\, \dd W_n(t)\,
\bigl(1+o_\Pb(1)\bigr).
\end{align*}

Taking into account the Lipschitz continuity of $\tilde\psi$, the inequality
$\lambda(\hat\vartheta,t)\geq\lambda_0$ and the fact that $\varepsilon_n
\sqrt{n}\, \varphi_n \longrightarrow 0$, the last term clearly converges to
zero in probability.

So, $A_n(u)$ has the same limit as
\begin{align*}
\tilde A_n(u) &= \varepsilon_n \sqrt{n} \int_{0}^{\tau}
\frac{S\,\bigl(\frac{t-\hat\vartheta-u\varphi_n}{\delta}\bigr)_+^\kappa -
  S\,\bigl(\frac{t-\hat\vartheta}{\delta}\bigr)_+^\kappa}{\lambda(\hat\vartheta,t)}\,\,
\dd W_n(t)\\*
&= \varepsilon_n \sqrt{n}\, \varphi_n^{\kappa+1/2}\,
\frac{S\sqrt{\lambda_*(\vartheta_0,\hat\vartheta)}}{\delta^\kappa}
\int_{-\hat\vartheta/\varphi_n}^{(\tau-\hat\vartheta)/\varphi_n}
\frac{(v-u)_+^\kappa -
  v_+^\kappa}{\lambda(\hat\vartheta,\hat\vartheta+v\varphi_n)}\,\, \dd
w_n(v)\\
&= \varepsilon_n \sqrt{n}\, \varphi_n^{\kappa+1/2}\,
\frac{S\sqrt{\lambda_*(\vartheta_0,\hat\vartheta)}}{\lambda_0\,
  \delta^\kappa}\,
\int_{-\hat\vartheta/\varphi_n}^{(\tau-\hat\vartheta)/\varphi_n}
\bigl[(v-u)_+^\kappa - v_+^\kappa\bigr]\, \dd w_n(v)\,
\bigl(1+o_\Pb(1)\bigr)\\*
&= \Gamma_\kappa^{-1}
\int_{-\hat\vartheta/\varphi_n}^{(\tau-\hat\vartheta)/\varphi_n}
\bigl[(v-u)_+^\kappa - v_+^\kappa\bigr]\, \dd w_n(v)\,
\bigl(1+o_\Pb(1)\bigr),
\end{align*}
where we used the change of variable $t=\hat\vartheta+v\varphi_n$ and denoted
\[
w_n(v)=\frac{W_n(\hat\vartheta+v\varphi_n) -
  W_n(\hat\vartheta)}{\sqrt{\lambda_*(\vartheta_0,\hat\vartheta) \varphi_n}}\,
.
\]

Therefore, noting that $w_n(\cdot) \Longrightarrow W(\cdot)$, where $W$ is a
two-sided Wiener process, we obtain
\[
A_n(u) \Longrightarrow \Gamma_\kappa^{-1} \int_{-\infty}^{+\infty}
\bigl[(v-u)_+^\kappa - v_+^\kappa\bigr]\, \dd W(v) = W^H(u),
\]
and hence
\[
\ln\hat Z_n(u) \Longrightarrow W^H(u) - \frac{u^2}{2} = \ln\hat Z(u),
\]
which yields the convergence of one-dimensional distributions of $\hat
Z_n(\cdot)$ to those of~$\hat Z(\cdot)$.  Clearly, the convergence of
multi-dimensional distributions equally holds.
\end{proof}

\begin{proof}[Proof of Lemma~\ref{L2}]
Throughout the proof, $c$ and $C$ denote ``generic'' strictly positive
constants, which can vary from formula to formula (and even in the same
formula).

Using the Taylor-Lagrange formula
\[
y^\varepsilon = 1 + \varepsilon\ln y + \frac{\varepsilon^2}{2}\, (\ln y)^2\,
y^{\gamma\varepsilon},
\]
where $y>0$ and $0<\gamma<1$, and denoting for shortness 
\[
D = \frac{\varepsilon_n^2}{8}\,
\biggl[\ln\biggl(\frac{\lambda(\vartheta_u,t)}{\lambda(\hat\vartheta,t)}\biggr)\biggr]^2
\biggl(\frac{\lambda(\vartheta_u,t)}{\lambda(\hat\vartheta,t)}\biggr)^{\gamma\varepsilon_n/2},
\]
we can write
\begin{align*}
\Ex\hat Z_n^{1/2}(u) &=  \Ex\exp\Biggl\{\frac{\varepsilon_n}{2}
\sum_{j=1}^{n} \int_{0}^{\tau}
\ln\biggl(\frac{\lambda(\vartheta_u,t)}{\lambda(\hat\vartheta,t)}\biggr)\,
\dd X_j(t) - n\, \frac{\varepsilon_n}{2} \int_{0}^{\tau}
\bigl[\lambda(\vartheta_u,t) \!-\! \lambda(\hat\vartheta,t)\bigr]\, \dd
t\Biggr\}\\*
&= \exp\Biggl\{n \int_{0}^{\tau} \biggl[
  \biggl(\frac{\lambda(\vartheta_u,t)}{\lambda(\hat\vartheta,t)}\biggr)^{\varepsilon_n/2}
  \!- 1\biggr]\lambda_*(\vartheta_0,t)\, \dd t - n\, \frac{\varepsilon_n}{2}
\int_{0}^{\tau} \bigl[\lambda(\vartheta_u,t) \!-\!
  \lambda(\hat\vartheta,t)\bigr]\, \dd t\Biggr\}\\
&= \exp\biggl\{n \int_{0}^{\tau} \biggl[ \frac{\varepsilon_n}{2}\,
  \ln\biggl(\frac{\lambda(\vartheta_u,t)}{\lambda(\hat\vartheta,t)}\biggr)
  \!+\!  D\biggr]\lambda_*(\vartheta_0,t)\, \dd t - n\,
\frac{\varepsilon_n}{2} \int_{0}^{\tau} \bigl[\lambda(\vartheta_u,t) \!-\!
  \lambda(\hat\vartheta,t)\bigr]\, \dd t\Biggr\}\\
&= \exp\biggl\{-n\, \frac{\varepsilon_n}{2} \bigl[\JKL(\vartheta_u) -
  \JKL(\hat\vartheta)\bigr] + n \int_{0}^{\tau} D\, \lambda_*(\vartheta_0,t)\,
\dd t\biggr\}\\*
&= \exp\bigl\{-F_n(u)+G_n(u)\bigr\}
\end{align*}
with evident notations.

For the first term, using~\eqref{borneJ}, we obtain
\[
-F_n(u) \leq -n\, \frac{\varepsilon_n}{2}\, \frac{m}{2}\,
(\vartheta_u-\hat\vartheta)^2 = -c\, u^2\, n\varepsilon_n\varphi_n^2 = -c\,
u^2.
\]

For the second term, we have
\begin{align}
G_n(u) &= n\, \frac{\varepsilon_n^2}{8} \int_{0}^{\tau}
\biggl[\ln\biggl(\frac{\lambda(\vartheta_u,t)}{\lambda(\hat\vartheta,t)}\biggr)\biggr]^2
\biggl(\frac{\lambda(\vartheta_u,t)}{\lambda(\hat\vartheta,t)}\biggr)^{\gamma\varepsilon_n/2}
\lambda_*(\vartheta_0,t)\, \dd t \notag\\*
&\leq C\, n\varepsilon_n^2 \int_{0}^{\tau}
\bigl[\ln\bigl(\lambda(\vartheta_u,t)\bigr) -
  \ln\bigl(\lambda(\hat\vartheta,t)\bigr)\bigr]^2\, \dd t \notag\\
&= C\, n\varepsilon_n^2 \int_{0}^{\tau} \biggl[\frac{\lambda(\vartheta_u,t) -
    \lambda(\hat\vartheta,t)}{\tilde\lambda}\biggr]^2\, \dd t \notag\\
& \leq C\, n\varepsilon_n^2 \int_{0}^{\tau} \bigl[\lambda(\vartheta_u,t) -
  \lambda(\hat\vartheta,t)\bigr]^2\, \dd t \notag\\
&\leq C\, n\varepsilon_n^2\,
\bil|\vartheta_u-\hat\vartheta\bir|^{2\kappa+1}\label{int2}\\*
& =C\, \bil|u\bir|^{2\kappa+1}\, n\varepsilon_n^2\varphi_n^{2\kappa+1} = C\,
\bil|u\bir|^{2\kappa+1}. \notag
\end{align}
Here $\tilde\lambda$ is some intermediate value between
$\lambda(\vartheta_u,t)$ and $\lambda(\hat\vartheta,t)$, we use the fact that
the intensities $\lambda$ and $\lambda_*$ are bounded and separated from zero,
and the inequality~\eqref{int2} is a particular case of the following more
general inequality (which will be also needed in the proof of the next lemma):
for any $p\geq 1$, there exist a constant $C$ such that
\begin{equation}
\label{int2p}
\int_{0}^{\tau} \bigl|\lambda(\vartheta_1,t) -
\lambda(\vartheta_2,t)\bigr|^{2p}\, \dd t \leq C\,
\bil|\vartheta_1-\vartheta_2\bir|^{2p\kappa+1}
\end{equation}
for all $\vartheta_1,\vartheta_2\in\Theta$.

Before continuing the proof of the lemma, let us prove the
inequality~\eqref{int2p}.  Without loss of generality we can suppose
that~$\vartheta_1>\vartheta_2$.  Using the elementary inequality
\begin{equation}
\label{ineqPow}
\bil|a+b\bir|^q \leq 2^{q-1} \bigl[\bil|a\bir|^q + \bil|b\bir|^q\bigr]
\end{equation}
(valid for all $a,b\in\RR$ and~$q\geq 1$), the Lipschitz continuity of
$\tilde\psi$ and the change of variable $t=\vartheta_2 + v\,
(\vartheta_1-\vartheta_2)$, we get
\begin{align*}
\int_{0}^{\tau} \bigl|\lambda(\vartheta_1,t) -
\lambda(\vartheta_2,t)\bigr|^{2p}\, \dd t &\leq C \int_{0}^{\tau} \biggl|S\,
\biggl(\frac{t-\vartheta_1}{\delta}\biggr)_+^\kappa - S\,
\biggl(\frac{t-\vartheta_2}{\delta}\biggr)_+^\kappa\biggr|^{2p}\, \dd t\\*
&\phantom{\leq{}} + C \int_{0}^{\tau} \bigl|\tilde\psi(t-\vartheta_1) -
\tilde\psi(t-\vartheta_2)\bigr|^{2p}\, \dd t\\
&\leq C \int_{0}^{\tau} \bigl|(t-\vartheta_1)_+^\kappa -
(t-\vartheta_2)_+^\kappa\bigr|^{2p}\, \dd t + C\,
(\vartheta_1-\vartheta_2)^{2p}\\*
&= C\, (\vartheta_1\!-\!\vartheta_2)^{2p\kappa+1}
\int_{-\frac{\vartheta_2}{\vartheta_1-\vartheta_2}}^{\frac{\tau-\vartheta_2}{\vartheta_1-\vartheta_2}}
\bigl|(v-1)_+^\kappa - v_+^\kappa\bigr|^{2p}\, \dd v + C\,
(\vartheta_1\!-\!\vartheta_2)^{2p}.
\end{align*}
As $\kappa<1/2$ and $p\geq 1$, we have
\[
\int_{-\frac{\vartheta_2}{\vartheta_1-\vartheta_2}}^{\frac{\tau-\vartheta_2}{\vartheta_1-\vartheta_2}}
\bigl|(v-1)_+^\kappa  - v_+^\kappa\bigr|^{2p}\, \dd v \leq
\int_{-\infty}^{+\infty} \bigl|(v-1)_+^\kappa - v_+^\kappa\bigr|^{2p}\, \dd v
< +\infty
\]
and (noting that $\vartheta_1-\vartheta_2 \leq \tau-\delta$ and $2p-2p\kappa-1 > 0$)
\[
C\, (\vartheta_1-\vartheta_2)^{2p} = C\,
(\vartheta_1-\vartheta_2)^{2p\kappa+1}
(\vartheta_1-\vartheta_2)^{2p-2p\kappa-1} \leq C\,
(\vartheta_1-\vartheta_2)^{2p\kappa+1},
\]
which yields the inequality~\eqref{int2p}.

Now, combining the bounds obtained for $-F_n(u)$ and $G_n(u)$, we have
\begin{equation}
\label{ineqL2bis}
\Ex \hat Z_n^{1/2}(u) \leq \exp\bigl\{-c u^2 + C
\bil|u\bir|^{2\kappa+1}\bigr\}.
\end{equation}
This concludes the proof of the lemma, since taking $c'=c/2$ and noting that
the function $-\frac{c}{2} u^2 + C \bil|u\bir|^{2\kappa+1}$ is bounded, we
obtain~\eqref{ineqL2}.

Note also that the moments $\Ex \hat Z_n^q(u)$, $u\in\UU_n$, of an arbitrary
order~$q>0$ can be bounded by the same inequalities~\eqref{ineqL2}
and~\eqref{ineqL2bis} (with constants depending on~$q$). Indeed, it is clear
from the proof above that only the order of the rate at which $\varepsilon_n
\longrightarrow 0$ is important, and so it is sufficient to apply the lemma to
$\tilde Z_n(u) = \bigl[Z_n(u)\bigr]^{\varepsilon'_n}$ with $\varepsilon'_n =
2q\varepsilon_n$, and note that $\tilde Z_n^{1/2}(u) =
\bigl[Z_n(u)\bigr]^{q\varepsilon_n} = \hat Z_n^q(u)$.  In particular, for any
$q>0$, there exist some constants~$c'=c'(q)>0$ and $C'=C'(q)>0$ such that
\begin{equation}
\label{ineqL2ter}
\Ex \hat Z_n^q(u) \leq C'\exp\bigl\{-c' u^2\bigr\}
\end{equation}
for all $n\in\NN$ and $u\in\UU_n$.
\end{proof}

\begin{proof}[Proof of Lemma~\ref{L3}]
Throughout the proof, once more $c$ and $C$ denote ``generic'' strictly
positive constants, which can vary from formula to formula (and even in the
same formula).

First of all, let us note that in the case $\bil|u_1-u_2\bir|\geq 1$,
using the inequality~\eqref{ineqL2ter}, we get
\[
\Ex \bigl[\hat Z_n^{1/2}(u_1) - \hat Z_n^{1/2}(u_2)\bigr]^2 \leq 2\, \Ex \hat
Z_n(u_1) + 2\, \Ex \hat Z_n(u_2) \leq C \leq C\, \bil|u_1 - u_2\bir|^\gamma.
\]
Hence, we can suppose from now on that $\bil|u_1-u_2\bir|\leq 1$ and, without
loss of generality, that $\bil|u_1\bir| \geq \bil|u_2\bir|$ (and, henceforth,
$\bil|u_1\bir| \leq \bil|u_2\bir|+1$).

Using the elementary inequality
\[
\bil|e^x - e^y\bir| \leq |x - y\bir|\, \max\{e^x,e^y\}
\]
(valid for all $x,y\in\RR$), we obtain
\[
\Ex \bigl[\hat Z_n^{1/2}(u_1) - \hat Z_n^{1/2}(u_2)\bigr]^2 \leq \Ex \Bigl(
\bigl|\ln\hat Z_n^{1/2}(u_1) - \ln\hat Z_n^{1/2}(u_2)\bigr|^2 \max\bigl\{\hat
Z_n(u_1),\hat Z_n(u_2)\bigr\}\Bigr).
\]

Now, let us fix some $p>1$ (the choice of $p$ will be precised later) and put
$q=\frac{p}{p-1}>1$ (so that $\frac1p+\frac1q=1$).  Using the H\"older
inequality, we can write
\begin{align*}
\Ex \bigl[\hat Z_n^{1/2}(u_1) - \hat Z_n^{1/2}(u_2)\bigr]^2 &\leq \Bigl[\Ex
\bigl|\ln\hat Z_n^{1/2}(u_1) - \ln\hat
Z_n^{1/2}(u_2)\bigr|^{2p}\Bigr]^{\frac1p}\, \Bigl[\Ex \max\bigl\{\hat
Z_n^q(u_1),\hat Z_n^q(u_2)\bigr\}\Bigr]^{\frac1q}\\*
&\leq \Bigl[\Bigl(\frac{\varepsilon_n}{2}\Bigr)^{2p}\, \Ex \bigl|\ln Z_n(u_1)
  - \ln Z_n(u_2)\bigr|^{2p}\Bigr]^{\frac1p}\, \Bigl[\Ex \bigl(\hat Z_n^q(u_1)
+ \hat Z_n^q(u_2)\bigr)\Bigr]^{\frac1q}\\
&\leq C \Bigl[\varepsilon_n^{2p}\, \Ex \bigl|\ln Z_n(u_1) - \ln
Z_n(u_2)\bigr|^{2p}\Bigr]^{\frac1p}\, \Bigl[e^{-c u_1^{2}} + e^{-c
  u_2^{2}}\Bigr]^{\frac1q}\\*
&\leq C e^{-c u_2^{2}}\, \Bigl[\varepsilon_n^{2p}\, \Ex \bigl|\ln Z_n(u_1) -
  \ln Z_n(u_2)\bigr|^{2p}\Bigr]^{\frac1p},
\end{align*}
where we used again the inequality~\eqref{ineqL2ter}.

Introducing a centered Poisson process of intensity function
$n\lambda_*(\vartheta_0,t)$, $t\in[0,\tau]$, by
\[
\pi_n(t) = \sum_{j=1}^{n} X_j(t) - n\int_{0}^{t} \lambda_*(\vartheta_0,s)\,
\dd s,
\]
we can write
\begin{align*}
\ln Z_n(u_1) - \ln Z_n(u_2) &= \sum_{j=1}^{n} \int_{0}^{\tau}
\ln\biggl(\frac{\lambda(\vartheta_{u_1},t)}{\lambda(\vartheta_{u_2},t)}
\biggr)\, \dd X_j(t) - n\int_{0}^{\tau} \bigl[\lambda(\vartheta_{u_1},t) -
  \lambda(\vartheta_{u_2},t)\bigr]\, \dd t\\*
&= \int_{0}^{\tau}
\ln\biggl(\frac{\lambda(\vartheta_{u_1},t)}{\lambda(\vartheta_{u_2},t)}
\biggr)\, \dd \pi_n(t)\\*
&\phantom{={}} - n \int_{0}^{\tau} \biggl[\lambda(\vartheta_{u_1},t) -
  \lambda(\vartheta_{u_2},t) - \lambda_*(\vartheta_0,t)
  \ln\biggl(\frac{\lambda(\vartheta_{u_1},t)}{\lambda(\vartheta_{u_2},t)}
  \biggr) \biggr] \dd t\\
&= \int_{0}^{\tau}
\ln\biggl(\frac{\lambda(\vartheta_{u_1},t)}{\lambda(\vartheta_{u_2},t)}
\biggr)\, \dd \pi_n(t) - n \bigl[\JKL(\vartheta_{u_1}) -
  \JKL(\vartheta_{u_2})\bigr]\\*
&= A_n(u_1,u_2) - B_n(u_1,u_2)
\end{align*}
with evident notations.  Therefore, using the inequality \eqref{ineqPow}, it
comes
\begin{align*}
\Ex \bigl[\hat Z_n^{1/2}(u_1) - \hat Z_n^{1/2}(u_2)\bigr]^2 &\leq C e^{-c
  u_2^{2}}\, \Bigl[\varepsilon_n^{2p}\, \Ex \bigl|A_n(u_1,u_2) -
  B_n(u_1,u_2)\bigr|^{2p}\Bigr]^{\frac1p}\\*
&\leq C e^{-c u_2^{2}}\, \Bigl[\varepsilon_n^{2p}\, \Ex
  \bigl|A_n(u_1,u_2)\bigr|^{2p} + \varepsilon_n^{2p}\,
  \bigl|B_n(u_1,u_2)\bigr|^{2p}\Bigr]^{\frac1p}.
\end{align*}

For the term containing $B_n(u_1,u_2)$, using the mean value theorem and the
upper bound of~\eqref{borneJ'}, we get
\begin{align*}
\varepsilon_n^{2p}\, \bigl|B_n(u_1,u_2)\bigr|^{2p} &= \Bigl| n\varepsilon_n
\bigl[\JKL(\vartheta_{u_1}) - \JKL(\vartheta_{u_2})\bigr]\Bigr|^{2p} = \bigl|
n\varepsilon_n\, (\vartheta_{u_1}-\vartheta_{u_2})\, \JKL'(\vartheta_{\tilde
  u})\bigr|^{2p}\\*
&\leq C\, \bigl| n\varepsilon_n\varphi_n\, (u_1-u_2)\, (\vartheta_{\tilde
  u}-\hat\vartheta)\bigr|^{2p} = C\, \bigl| n\varepsilon_n\varphi_n^2\,
(u_1-u_2)\, \tilde u\bigr|^{2p}\\*
&\leq C\, \bil|u_1-u_2\bir|^{2p}\,
\bigl(\max\{\bil|u_1\bir|,\bil|u_2\bir|\}\bigr)^{2p} \leq C\,
\bil|u_1-u_2\bir|^{2p}\, (1+\bil|u_2\bir|)^{2p}.
\end{align*}
Here $\tilde u$ is some intermediate value between $u_1$ and $u_2$.

For the term containing $A_n(u_1,u_2)$, using Rosenthal's inequality (see, for
example,~\cite{Wood99}) and proceeding similarly as while bounding $G_n(u)$ in
the proof of the previous lemma, we have
\begin{align*}
\varepsilon_n^{2p}\, \Ex \bigl|A_n(u_1,u_2)\bigr|^{2p} &\leq C\,
\varepsilon_n^{2p}\, \biggl(n \int_{0}^{\tau}
\biggl[\ln\biggl(\frac{\lambda(\vartheta_{u_1},t)}{\lambda(\vartheta_{u_2},t)}\biggr)\biggr]^2
\lambda_*(\vartheta_0,t)\, \dd t\biggr)^p\\*
&\phantom{\leq{}} + C\, n\varepsilon_n^{2p}\, \int_{0}^{\tau}
\biggl|\ln\biggl(\frac{\lambda(\vartheta_{u_1},t)}{\lambda(\vartheta_{u_2},t)}\biggr)\biggr|^{2p}
\lambda_*(\vartheta_0,t)\, \dd t\\
&\leq C\, \biggl(n\varepsilon_n^2 \int_{0}^{\tau}
\bigl[\lambda(\vartheta_{u_1},t)-\lambda(\vartheta_{u_2},t)\bigr]^2\, \dd
t\biggr)^p\\*
&\phantom{\leq{}} + C\, n\varepsilon_n^{2p}\, \int_{0}^{\tau}
\bigl|\lambda(\vartheta_{u_1},t)-\lambda(\vartheta_{u_2},t)\bigr|^{2p}\, \dd
t\\
&\leq C\, \bigl(n\varepsilon_n^2\,
\bil|\vartheta_{u_1}-\vartheta_{u_2}\bir|^{2\kappa+1}\bigr)^p + C\,
n\varepsilon_n^{2p}\, \bil|\vartheta_{u_1}-\vartheta_{u_2}\bir|^{2p\kappa+1}\\*
&\leq C\, \bil|u_1-u_2\bir|^{(2\kappa+1)p} + C\,
\bil|u_1-u_2\bir|^{2p\kappa+1}.
\end{align*}
Here we equally used the inequality~\eqref{int2p}, the fact that
$n\varepsilon_n^2\varphi_n^{2\kappa+1}=C$ and the boundedness
of~$n\varepsilon_n^{2p}\varphi_n^{2p\kappa+1}
=o(n\varepsilon_n^2\varphi_n^{2\kappa+1}) = o(1)$.

So, finally, we obtain
\begin{align*}
\Ex \bigl[\hat Z_n^{1/2}(u_1) - \hat Z_n^{1/2}(u_2)\bigr]^2 &\leq C e^{-c
  u_2^{2}}\, \Bigl[ \bil|u_1-u_2\bir|^{(2\kappa+1)p} +
  \bil|u_1-u_2\bir|^{2p\kappa+1}\\
&\phantom{\leq C e^{-c u_2^{2}}\, \Bigl[} {} + \bil|u_1-u_2\bir|^{2p}\,
    (1+\bil|u_2\bir|)^{2p}\Bigr]^{\frac1p}\\
&\leq C e^{-c u_2^{2}}\, \bil|u_1-u_2\bir|^{2\kappa+\frac1p}\,
  (1+\bil|u_2\bir|)^2\\*
&\leq C\, \bil|u_1-u_2\bir|^{2\kappa+\frac1p},
\end{align*}
since the function $u \longmapsto e^{-c u^2}\, (1+\bil|u\bir|)^2$ is bounded.

To conclude the proof of the lemma, it remains to notice that choosing $p>1$
sufficiently close to~$1$, we can make $\gamma = 2\kappa+\frac1p < 2\kappa+1$
arbitrary close to $2\kappa+1$ and, in particular, strictly grater than~$1$.
\end{proof}

\section{Discussion}

Recall that if we have a cusp-type singularity of order $\kappa \in (0,1/2)$
and there is no misspecification, the mean square error of the MLE has the
following asymptotics (see~\cite{D03}):
\[
\Ex \bigl(\hat\vartheta_n-\vartheta\bigr)^2 = c\, n^{-\frac{2}{2\kappa+1}}
\bigl(1+o(1)\bigr).
\]
Therefore, the smaller is the value of $\kappa$, the better is the rate of
convergence. It is interesting to compare this rate with the rate of
convergence for the model with misspecification. According to
Theorem~\ref{T1}, the corresponding mean square error is
\[
\Ex \bigl(\hat\vartheta_n-\vartheta\bigr)^2 = c\, n^{-\frac{2}{3-\kappa}}
\bigl(1+o(1)\bigr),
\]
and so we have an opposite situation: the smaller is the value of $\kappa$,
the worse is the rate of convergence.

The plots of the rate exponents $\gamma = \frac{2}{2\kappa +1}$ and $\gamma =
\frac{2}{3-\kappa}$ with and without misspecification are given in
Fig.~\ref{F5}.  Note that for $\kappa > 1/2$ (regular case), the plotted value
is~$\gamma=1$, since in this case the mean square error goes to zero at
rate~$1/n$ both with and without misspecification (see~\cite{Kut22}).

\begin{figure}[H]
\begin{center}
\includegraphics[width=0.5\textwidth]{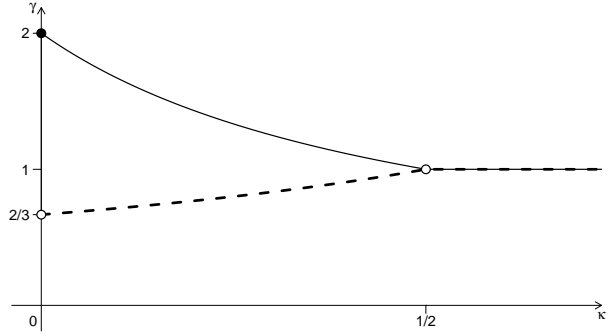}
\caption{Rate exponents $\gamma$ with (dashed line) and without (solid line)
  misspecification}
\label{F5}
\end{center}
\end{figure} 

The limit at $\kappa=0$ of the solid line corresponds well to the rate
exponent $\gamma=2$ of the change-point problem with discontinuous intensity
function (see~\cite{Kut98}). In the case of misspecification the situation is
essentially different. If the intensity function is discontinuous, then the
p-MLE converges to the true value (is consistent) and $\gamma = 2$
(see~\cite{Kut22}), while the limit at $\kappa=0$ of the dashed curve is
only~$2/3$.

%Note that this case ($\kappa=0$) was not included in the considered in this
%work study, but in the work~\cite{AD21} it was proved that the function
%$\gamma(\kappa)$, $\kappa \in \bigl(0,\frac12\bigr)$ converges to 2 as $\kappa
%\longrightarrow 0$.

Note that the case $\kappa=1/2$ was not included in this study. If there is no
misspecification and $\kappa =1/2$, we are in the \emph{almost smooth\/} case,
and the error is
\[
\Ex \bigl(\hat\vartheta _n-\vartheta\bigr)^2 = \frac{c}{n\ln n}\,
\bigl(1+o(1)\bigr)
\]
(see \cite{IH81,Kut22}). The properties of the p-MLE for the model with
misspecification and $\kappa =1/2$ were not yet studied. Of course, this can
be done with the help of the developed in this work approach.

Note also that it is possible to generalize the presented in this work results
to the case of non constant signals, i.e., when the theoretical and real
intensity functions of the observed inhomogeneous Poisson processes are given
by
\begin{align*}
\lambda(\vartheta,t) & = S(t)\, \psi(t-\vartheta)+\lambda _0,\\*
\lambda_* (\vartheta,t) &= S_*(t)\, \psi(t-\vartheta)+\lambda _0,
\end{align*}
with functions $S(\cdot)$ and $S_*(\cdot)$ satisfying the condition
\[
\inf_{t \in \Theta}\bigl|S(t)-S_*(t)\bigr|>0.
\]
The main difference will be in the proof of Proposition~\ref{P1}.

\bigskip
\noindent
\textbf{Acknowledgments.} This research was financially supported by the
Ministry of Education and Science of the Russian Federation project
no.~FSWF-2023-0012 (sections~1 and~2) and by the Russian Science Foundation
project no.~20-61-47043 (sections~3 and~4).

\end{document}